\documentclass[11pt]{extarticle}
\usepackage{graphicx}
\usepackage{amsthm}
\usepackage{amsmath}
\usepackage{amsfonts}
\usepackage[numbers]{natbib}
\usepackage{fullpage}
\usepackage{amssymb}
\usepackage{graphics}
\usepackage{mathdots}
\usepackage{enumerate}
\usepackage{hyperref}
\usepackage{tikz}
\usepackage{array}

\usepackage{pgfplots}
\pgfplotsset{compat=1.12}
\usepgfplotslibrary{fillbetween}
\usetikzlibrary{patterns}

\newtheorem{thm}{Theorem}[section]
\newtheorem{lem}[thm]{Lemma}        
\newtheorem{cor}[thm]{Corollary}
\newtheorem{deff}[thm]{Definition}

\newcommand{\gnm}{\mathcal{G}_{n,m}}

\begin{document}

\title{\bf On maximum graphs in Tutte polynomial posets}

\date{}      
\maketitle

\vspace{-40pt}

\noindent
\[\text{Nathan Kahl}\,^\mathrm{1}, \text{Kristi Luttrell}\,^\mathrm{1}\]

\bigskip
\noindent

{\footnotesize $^{\mathrm{1}}$\,\textit{Department of Mathematics and
    Computer Science, Seton Hall University, South Orange, NJ 07079,
    USA}}\\[0.5mm]
{\footnotesize Email: 
  \url{kahlnath@shu.edu}, \url{kristi.luttrell@shu.edu}}

\bigskip
\begin{abstract}
  \noindent
Boesch, Li, and Suffel were the first to identify the existence of uniformly optimally reliable graphs (UOR graphs), graphs which maximize all-terminal reliability over all graphs with $n$ vertices and $m$ edges.  The all-terminal reliability of a graph, and more generally a graph's all-terminal reliability polynomial $R(G;p)$, may both be obtained via the Tutte polynomial $T(G;x,y)$ of the graph $G$.  Here we show that the UOR graphs found earlier are in fact maximum graphs for the Tutte polynomial itself, in the sense that they are maximum not just for all-terminal reliability but for a vast array of other parameters and polynomials that may be obtained from $T(G;x,y)$ as well.  These parameters include, but are not limited to, enumerations of a wide variety of well-known orientations, partial orientations, and fourientations of $G$; the magnitudes of the coefficients of the chromatic and flow polynomials of $G$; and a wide variety of generating functions, such as generating functions enumerating spanning forests and spanning connected subgraphs of $G$.  The maximality of all of these parameters is done in a unified way through the use of $(n,m)$ Tutte polynomial posets.
\end{abstract}

\section{Introduction}
\label{sec1}

The \textit{all-terminal reliability} of a connected graph $G$, denoted $R(G;p)$, is the probability that $G$ remains connected when edges fail independently with probability $p$.  It is well-known that $R(G;p)$ can be written as a polynomial in terms of $p$ whose coefficients carry important combinatorial information about $G$, see for example \cite{c87}.
Let $\gnm$ denote the class of connected  graphs with $n$ vertices and $m$ edges.  It is possible for the graphs of the all-terminal reliability polynomials of two distinct graphs $G,H \in \gnm$ to intersect on the interval $p \in (0,1)$, once or even multiple times \cite{chm93,k00}.  In terms of network strength, this means that when comparing two network topologies that use the same number of nodes and links, it may be the case that neither network is uniformly best:  which network is stronger may depend upon the value of $p$.  As a consequence, for some values of $n,m$ there may be no best network topology, in other words there may be no graph $H \in \gnm$ such that $R(H;p) \ge R(G;p)$ for all $G \in \gnm$ and all $p \in (0,1)$.  This phenomenon has in fact been shown to occur for various values of $n,m$ \cite{bc14,k81,mcpp91}.

However in 1991, Boesch, Li, and Suffel were the first to determine that there were $\gnm$ classes for which \textit{uniformly optimally reliable graphs} (or more simply \textit{UOR graphs}) existed, that is, there were classes $\gnm$ and graphs $H \in \gnm$ for which  $R(H;p) \ge R(G;p)$ for all $G \in \gnm$ and all $p \in (0,1)$.  In terms of networks, UOR graphs are the best network topologies on $n$ nodes and $m$ links, regardless of edge failure probability $p$.  Since then other $\gnm$ classes with UOR graphs have been identified, and in fact a rich literature has grown up around UOR graphs and all-terminal reliability in general, see for example the surveys \cite{ab84,bss09,bccgm20,c87,r21s}.

The Tutte polynomial $T(G;x,y)$ of a graph $G$ is a multivariate polynomial which encodes a vast array of structural features of the graph $G$.  One of these structural features is the all-terminal reliability $R(G;p)$.  We will say more about the definition and computation of $T(G;x,y)$ later, for now we give some idea of the scope of these structural features, beyond all-terminal reliability.  Various evaluations of $T(G;x,y)$ have been shown to give the number of spanning trees, spanning forests, and spanning connected subgraphs of $G$ \cite{em11};  the number of various types of orientations of $G$, including acyclic orientations \cite{s73}, totally cyclic orientation \cite{gz83}, and acyclic orientations with a single source \cite{l77}; various types of partial orientations of $G$, where edges may be either oriented or unoriented \cite{b18,gs96}; and various types of fourientations of $G$, where edges may be either oriented, unoriented, or bidirectionally oriented \cite{bh17}.  The Tutte polynomial $T(G;x,y)$ can also be specialized to various other well-known graph polynomials of $G$, including the all-terminal reliability polynomial $R(G;p)$, but others here include the chromatic polynomial and flow polynomial \cite{em11,w99}; various generating functions including generating functions for spanning forests of $i$ components, and spanning connected subgraphs of $i$ edges \cite{em11}; generating functions relating to orientations, partial orientations, and fourientations of a graph \cite{gs96,bh17}; and the generating function for the number of critical configurations of level $i$ of the Abelian sandpile model \cite{m97}.  
In general, $T(G;x,y)$ has been shown to encode any graph invariant which obeys a deletion-contraction law, a property sometimes called the universality property of the Tutte polynomial.
For more information on the Tutte polynomial and its applications we refer the reader to the surveys \cite{bo92,em11,ow79,w99}.

Given that such a wide variety of graph parameters and polynomials have the Tutte polynomial as a common generalization, it is natural to ask how strongly these evaluations and specializations may be related.  If $H \in \gnm$ is a UOR graph in $\gnm$, for example, does that mean that $H$ has more acyclic orientations than every other $G \in \gnm$?  Or vice-versa, does identifying an $H \in \gnm$ with the most acyclic orientations also serve to identify the UOR graph in that class?   Clearly many other relationships between Tutte parameters can be explored as well.

Unfortunately, although unsurprisingly, identifying an $H \in \gnm$ that maximizes one of the many Tutte polynomial parameters is no guarantee that $H$ maximizes others.  One of the smallest examples here occurs in the class $\mathcal{G}_{7,11}$.  Two graphs from that class are pictured in Figure \ref{fig0},
one of which is maximum for all-terminal reliability (i.e., is the UOR graph for the class) and also maximizes certain other Tutte polynomial parameters, while the other maximizes different Tutte polynomial parameters, for example acyclic orientations.\footnote{These graphs and other graphs in the class they belong to were checked using the \textit{Mathematica} program, whose {\tt TuttePolynomial} command can generate Tutte polynomials of specific graphs and whose {\tt GraphData} database contains all connected graphs on 7 or fewer vertices.}  

However, somewhat in parallel with the UOR graph phenomenon mentioned earlier, just because some graph classes do not have such a ``maximum'' graph for the Tutte polynomial does not mean that all graph classes do not.  In fact some graph classes contain graphs that are maximum for the Tutte polynomial for a remarkable variety of graph parameters, including all of the ones previously listed.  We are able to identify these Tutte-maximum graphs in a unified way by using the recently developed \textit{$(n,m)$ Tutte polynomial posets} \cite{k22}.  

An $(n,m)$ Tutte polynomial poset, abbreviated to \textit{Tutte poset} if the $\gnm$ class is arbitrary or clear, was shown in \cite{k22} to capture to a remarkable extent the behavior of graph parameters that may be obtained from the Tutte polynomial.  The poset is defined by the relation $G \preccurlyeq H$ if and only if 
\[ T(H;x,y) - T(G;x,y) = (x+y-xy)P(x,y)\]
for some polynomial $P(x,y)$ with non-negative coefficients.  
 Specifically, the relation defined above implies the following.

\begin{thm}[\cite{k22}]
Let $G,H \in \gnm$ be such that $G \preccurlyeq H$, in other words
\begin{equation}\label{tpdef} T(H;x,y) - T(G;x,y) = (x+y-xy)P(x,y)\end{equation}
where $P(x,y)$ is a polynomial with non-negative coefficients.  Then
\begin{enumerate}
\item  we have $\rho(G) \le \rho(H)$ for all of the Tutte polynomial parameters previously mentioned (spanning trees, etc.)
\item  we have $|c_i(G)| \le |c_i(H)|$ for any $i$,  where $c_i$ stands for the $i^{th}$ coefficient of the Tutte polynomial specializations previously mentioned (chromatic, etc.).
\end{enumerate}
\end{thm}
Thus $G \preccurlyeq H$ implies not just maximality of the all-terminal reliability, i.e., $R(H;p) \ge R(G;p)$ for all $p \in (0,1)$, but in fact implies each coefficient of $R(H;p)$ is larger in absolute value than the corresponding coefficient in $R(G;p)$.  And similar statements follow for other Tutte polynomial parameters and other polynomial specializations of the Tutte polynomial.  We note that  there are a number of additional Tutte polynomial parameters whose behavior are also captured by Tutte posets besides the ones mentioned here, as well as some Tutte polynomial parameters for which maximality in the Tutte poset implies minimality for the parameter; we refer the reader to \cite{k22} for more details.

Given the previously noted facts, we make the following definition.
\begin{deff}
A graph $H \in \gnm$ is \emph{Tutte-maximum} if $G \preccurlyeq H$ for all $G \in \gnm$.
\end{deff}
\noindent The purpose of this paper is to demonstrate that there exist infinite families of graph classes which contain Tutte-maximum graphs.  The classes in which we are able to demonstrate these maximum graphs are, in fact, the same edge-sparse classes in which Boesch, Li, and Suffel first demonstrated the existence of UOR graphs, the 
classes $\mathcal{G}_{n,n}$, $\mathcal{G}_{n,n+1}$ and $\mathcal{G}_{n,n+2}$.

The structure of the paper is as follows.  In the next section we present preliminary facts on the Tutte polynomial which are necessary for our results.  
In the next few sections we determine the Tutte-maximum graphs for $\mathcal{G}_{n,n}$, $\mathcal{G}_{n,n+1}$, and $\mathcal{G}_{n,n+2}$ in turn, noting the increasing difficulty of the task.  As mentioned, these Tutte-maximum results generalize and extend the results of Boesch, Li, and Suffel \cite{bls91} by showing that the graphs identified there are in fact the unique maximum graphs for far more Tutte polynomial parameters than just all-terminal reliability.   To conclude we conjecture, based partly on other work on UOR graphs, that the $\mathcal{G}_{n,n+3}$ graph class has a Tutte-maximum graph, and describe the form of these conjectured maximum graphs.  We note that the graph classes $\mathcal{G}_{n,n+4}$ and $\mathcal{G}_{n,n+5}$ do not have unique maximum graphs, at least not for all $n$, and describe the counterexamples found there.  We also conjecture the form of Tutte-maximum graphs in a few other graph classes based on previous results on spanning trees.

\begin{figure}  
\begin{center}
\resizebox{12cm}{!}{
\begin{tikzpicture}

\fill (-12.3:1.2) node[black,right] {} circle (0.1cm) ;
\fill (39.1:1.2) node[black,right] {} circle (0.1cm) ;
\fill (90.5:1.2) node[black,right] {} circle (0.1cm) ;
\fill (141.9:1.2) node[black,right] {} circle (0.1cm) ;
\fill (193.3:1.2) node[black,right] {} circle (0.1cm) ;
\fill (244.7:1.2) node[black,right] {} circle (0.1cm) ;
\fill (296.1:1.2) node[black,right] {} circle (0.1cm) ;
\draw (-12.3:1.2) -- (39.1:1.2) -- (90.5:1.2) -- (141.9:1.2) -- (193.3:1.2) -- (244.7:1.2) -- (296.1:1.2) -- (-12.3:1.2);
\draw (296.1:1.2) -- (90.5:1.2) -- (244.7:1.2);
\draw (-12.3:1.2) -- (141.9:1.2);
\draw (193.3:1.2) -- (39.1:1.2);

\begin{scope}[xshift=5cm]
\fill (0,-1.1) node[black,right]  {} circle (0.1cm) ;
\fill (0,0) node[black,right] {} circle (0.1cm) ;
\fill (0,1.1) node[black,right] {} circle (0.1cm) ;
\fill (2,-1.1) node[black,right] {} circle (0.1cm) ;
\fill (2,0) node[black,right] {} circle (0.1cm) ;
\fill (2,1.1) node[black,right] {} circle (0.1cm) ;
\fill (-1,0) node[black,right] {} circle (0.1cm) ;
\foreach \x in {-1.1,0,1.1} \draw (0,-1.1) -- (2, \x);
\foreach \y in {-1.1,0,1.1} \draw (0,0) -- (2, \y);
\foreach \z in {-1.1,0,1.1} \draw (0,1.1) -- (2, \z);
\draw (0,-1.1) -- (-1,0) -- (0,1.1);
\end{scope}

  \end{tikzpicture}
	}
\end{center}
\caption{The two maximal graphs of the $(7,11)$ Tutte polynomial poset.  The left graph maximizes all-terminal reliability, and the right graph maximizes acyclic orientations, in the $\mathcal{G}_{7,11}$ graph class.  Both parameters may be obtained from the Tutte polynomial.} 
\label{fig0}
\end{figure}
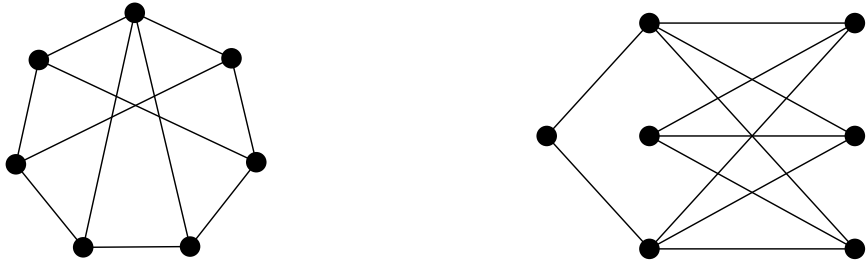

For any terms and notations left undefined in the paper we refer the reader to a standard reference like \cite{b02}.   In general our notation will be standard as well, although we note the following non-standard notations which we may employ.  If there is no confusion, we omit brackets around vertex or edge sets, writing for example $G-v$ (resp. $G-uv$) instead of $G-\{v\}$ (resp. $G-\{uv\}$) to indicate the graph $G$ with the vertex $v$ (resp. edge $uv$) removed.  In addition the following notation will be used extensively.  The Tutte polynomial is block-invariant, in other words, if graphs $G$ and $H$ have the same block structure then $T(G;x,y) = T(H;x,y)$.  When considering graphs that have multiple blocks, then, we introduce the notation $G= B_1 \cdot B_2$ to indicate any graph whose blocks are $B_1$ and $B_2$, or more precisely, any member of the Tutte poset equivalence class whose graphs consist of blocks $B_1$ and $B_2$.  (The `dot' is intended to suggest both the common vertex of $B_1$ and $B_2$ and the multiplication operation since Tutte polynomials multiply over blocks.)  Furthermore, when multiples of the same block appear we may abbreviate in the natural way; for example $2C_4 \cdot 3K_5$ stands for any graph whose five blocks consist of two 4-cycles and three complete graphs on 5 vertices.

\section{Preliminaries}

We begin with some necessary facts about the Tutte polynomial.  The Tutte polynomial $T(G;x,y)$ of a graph $G$ may be defined recursively using edge deletions and contractions as 
\begin{equation} \label{def} T(G;x,y) = \left\{\begin{array}{ll} T(G-e;x,y) + T(G/e;x,y) & \textrm{if $e$ is not a bridge or loop}\\
x T(G/e;x,y) & \textrm{if $e$ is a bridge}\\
y T(G-e;x,y) & \textrm{if $e$ is a loop} \end{array} \right. 
\end{equation}
with the Tutte polynomial of a graph with no edges equal to 1.  As mentioned in the Introduction, the Tutte polynomial factors over blocks, i.e., when $G= G_1 \cdot G_2 \cdot \ldots \cdot G_t$, then
\begin{equation} \label{fact}
T(G;x,y) = T(G_1;x,y) T(G_2;x,y) \dots T(G_t;x,y)
\end{equation}
regardless of how the blocks of $G$ are arranged.  Thus the Tutte polynomial of every tree on $n$ vertices is identically $x^{n-1}$, since any tree on $n$ vertices has $n-1$ bridges as its constituent blocks.  (In terms of Tutte posets, this shows that the $(n,n-1)$ Tutte polynomial poset is rather uninteresting, being the trivial poset with a single equivalence class, that equivalence class containing every tree on $n$ vertices.)

Using the deletion-contraction definition (\ref{def}) and factoring over blocks, it is easy to determine the Tutte polynomial of some basic graphs that will be useful later.  We mention two well-known graphs and their Tutte polynomials in particular.  Let $C_n$ denote the cycle on $n$ vertices, $n\geq2$, with the cycle $C_2$ being the multigraph consisting of two edges between two vertices.  We then have
\begin{equation} \label{cyc}
T(C_n;x,y) = y + x + \dots + x^{n-1}.
\end{equation}
where when $n=1$ we take $C_1$ to be a loop and $T(C_1;x,y) = y$ as in (\ref{def}).  With $M_m$ denoting a \emph{multiedge block}, a graph consisting of two vertices with $m\geq2$ edges between them, we have
\begin{equation} \label{mte}
T(M_m;x,y) = x + y + \dots + y^{m-1}.
\end{equation}
where when $m=1$ we take $M_1$ to be a bridge and $T(M_1;x,y) = x$ as in (\ref{def}).  All of the previous basic facts can be found, for example, in the surveys \cite{bo92,em11,ow79,w99}.  We also have the following useful generalization of (\ref{def}) above, found in \cite{hpr10}.  A \emph{$k$-ear} $E_k$ of a graph $G$ is an induced path of $k$ edges between two distinct vertices $u,v \in V(G)$.   A 1-ear is then just an edge between two vertices, while for $k \ge 2$ a $k$-ear may be viewed as a subdivided edge between two vertices.  There is therefore a natural notion of deleting and/or contracting a  $k$-ear $E_k$ , which we denote by $G-E_k$ and $G/E_k$ respectively:  by $G-E_k$ we mean deleting from $G$ all $k$ edges of $E_k$ along with the internal vertices of the path $E_k$, and by $G/E_k$ we mean deleting the ear $E_k$ and then identifying the two endvertices of $E_k$.
\begin{thm}[\cite{hpr10}]  \label{e1}
Let $G$ be a graph with $k$-ear $E_k$, $k \ge 1$, with all edges of $E_k$ non-bridges.  Then 
\[T(G;x,y) = (1+x+x^2+\dots+x^{k-1})T(G-E_k;x,y)+T(G/E_k;x,y).\]
\end{thm}

As mentioned in the Introduction, the $(n,m)$ Tutte polynomial poset is defined on $\gnm$ by the relation $G \preccurlyeq H$ if and only if 
\begin{equation}\label{tpdef} T(H;x,y) - T(G;x,y) = (x+y-xy)P(x,y)\end{equation}
for some polynomial $P(x,y)$ with non-negative coefficients.  
If $P(x,y) = 0$ then $T(G;x,y) = T(H;x,y)$ and $G$ and $H$ are \emph{Tutte polynomial equivalent} or \emph{$T$-equivalent} graphs.  The $(n,m)$ Tutte polynomial poset thus is more precisely a poset on the equivalence classes of the graphs of $\mathcal{G}_{n,m}$ with $G,H$ equivalent if and only if $T$-equivalent, but there will be no problem if we allow any representative of an equivalence class to represent the class as a whole.  If $P(x,y) \neq 0$ then we indicate this by $G \prec H$.   

Finally, the following technical lemma will be useful.  The proof, which is an algebraic exercise, is omitted.

\begin{lem} \label{tlem1}
Let $a,b$ be positive integers.  Then \[x + y\left(\sum_{i=0}^{a+b-3}y^i \right) - y^{a-1}\left(x + y\left(\sum_{i=0}^{b-2}y^i\right)\right) = (x+y-xy)\sum_{i=0}^{a-2}y^i
\] where it is understood that $\displaystyle \sum_{i=0}^{-1}y^i = 0$.
\end{lem}

\section{Tutte polynomial posets for $\mathcal{G}_{n,n}$}

Graphs in the class $\mathcal{G}_{n,n}$ are all \emph{unicyclic graphs}, graphs of the form $C_k \cdot (n-k) K_2$ for $3 \le k \le n$.  As such, the Tutte polynomial poset of $\mathcal{G}_{n,n}$ could be resolved using (\ref{def}), (\ref{fact}), and (\ref{cyc}) from the previous section, along with the necessary algebra.  Instead however we prove a more general result in Theorem \ref{comb} below, which effectively says that moving an edge in order to combine blocks will produce a ``better'' graph in the Tutte poset.  Theorem \ref{comb} will resolve $\mathcal{G}_{n,n}$ but will also be useful in resolving larger classes.

\begin{figure} 
\begin{center}
\resizebox{15cm}{!}{
\begin{tikzpicture}

\draw (-2,0) ellipse (2cm and 1.5cm);
\draw (2,0) ellipse (2cm and 1.5cm);
\fill (0,0) node[black,label={[yshift=-0.05cm,xshift=-0.3cm]{\Large $v$}}] {} circle (0.1cm);

\fill (1.3,0.5) node[black,label={[yshift=-0.1cm,xshift=0.3cm]{\Large $w$}}] {} circle (0.1cm);
\fill (1.3,0.1) node[black,right] {} circle (0.1cm);
\fill (1.3,-0.9) node[black,right] {} circle (0.1cm);
\fill (-1.3,0.5) node[black,label={[yshift=-0.1cm,xshift=-0.3cm]{\Large $u$}}] {} circle (0.1cm);
\fill (-1.3,0.1) node[black,right] {} circle (0.1cm);
\fill (-1.3,-0.9) node[black,right] {} circle (0.1cm);

\fill (1.3,-0.6) node[black,right] {} circle (0.03cm);
\fill (1.3,-0.2) node[black,right] {} circle (0.03cm);
\fill (1.3,-0.4) node[black,right] {} circle (0.03cm);
\fill (-1.3,-0.6) node[black,right] {} circle (0.03cm);
\fill (-1.3,-0.2) node[black,right] {} circle (0.03cm);
\fill (-1.3,-0.4) node[black,right] {} circle (0.03cm);

\draw (1.3,0.5) -- (0,0) -- (1.3,-0.9);
\draw (-1.3,0.5) -- (0,0) -- (-1.3,-0.9);
\draw (1.3,0.1) -- (0,0) -- (-1.3,0.1);

\begin{scope}[xshift=9cm]
\draw (-2,0) ellipse (2cm and 1.5cm);
\draw (2,0) ellipse (2cm and 1.5cm);
\fill (0,0) node[black,label={[yshift=-0.05cm,xshift=-0.3cm]{\Large $v$}}] {} circle (0.1cm);

\fill (1.3,0.5) node[black,label={[yshift=-0.1cm,xshift=0.3cm]{\Large $w$}}] {} circle (0.1cm);
\fill (1.3,0.1) node[black,right] {} circle (0.1cm);
\fill (1.3,-0.9) node[black,right] {} circle (0.1cm);
\fill (-1.3,0.5) node[black,label={[yshift=-0.1cm,xshift=-0.3cm]{\Large $u$}}] {} circle (0.1cm);
\fill (-1.3,0.1) node[black,right] {} circle (0.1cm);
\fill (-1.3,-0.9) node[black,right] {} circle (0.1cm);

\fill (1.3,-0.6) node[black,right] {} circle (0.03cm);
\fill (1.3,-0.2) node[black,right] {} circle (0.03cm);
\fill (1.3,-0.4) node[black,right] {} circle (0.03cm);
\fill (-1.3,-0.6) node[black,right] {} circle (0.03cm);
\fill (-1.3,-0.2) node[black,right] {} circle (0.03cm);
\fill (-1.3,-0.4) node[black,right] {} circle (0.03cm);

\draw (0,0) -- (1.3,0.5);
\path (1.3,0.5) edge [bend right] (-1.3,0.5); 
\draw (0,0) -- (1.3,-0.9);
\draw  (0,0) -- (-1.3,-0.9);
\draw (1.3,0.1) -- (0,0) -- (-1.3,0.1);

\end{scope}

  \end{tikzpicture}
}
\end{center}
\caption{An illustration of the graphs $G$ (left) and $H$ (right) from Theorem \ref{comb}.  Additional blocks, if present, are identical in the two graphs and are not shown.} 
\label{fig1}
\end{figure}
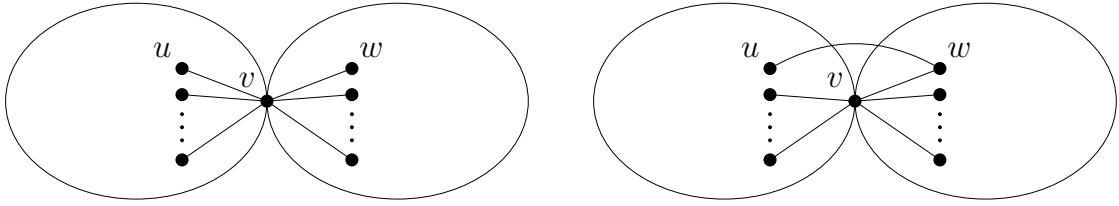

\begin{thm}  \label{comb}
Let $G_1,G_2$ be two blocks of the graph $G$, with $v$ the common cutvertex of $G_1,G_2$, and let $uv \in E(G_1)$ and $wv \in E(G_2)$.  Define $H = G- uv + uw$.  Then we have $G \preccurlyeq H$ and, if $G_1 \neq K_2$, we have $G \prec H$.
\end{thm}

An illustration of $G$ and $H$ appears in Figure \ref{fig1}.  Our proof approach is similar to the proof of one of the main theorems (Theorem 3.1) in \cite{k22}, although here the simple graphical structure and minimal change from $G$ to $H$ make the process much more straightforward.

\begin{proof}
Since the Tutte polynomial factors over blocks and is itself a polynomial with non-negative coefficients, we may assume that $G$ consists solely of the two blocks $G_1, G_2$ and that $H$ is the single block $G - uv + uw$.  
Furthermore, if $G_1= K_2$ then the block structure of $G$ and $H$ are identical.  In this case $T(G;x,y) - T(H;x,y) = 0$ and trivially we have $G \preccurlyeq H$.  Hence for the rest of the proof we assume that $G_1 \neq K_2$.

The graphs $G$ and $H$ differ only in one edge and so, abusing notation somewhat,  we will often use the same labels to identify corresponding edges of the graphs, as this helps identify corresponding terms in the decomposition (\ref{def}).  To begin, the subgraphs $G-v$ and $H - v - uw$ are identical and we let $S$ denote the edges of both subgraphs.  We now decompose $T(G;x,y)$ and $T(H;x,y)$ by ``deleting/contracting every edge of $S$''.  Formally, let $\{e_1,\dots,e_{\ell}\}$ be a fixed but arbitrary ordering of the edges of $S$, and let $R = (r_1,\dots,r_{\ell})$ be a binary vector in $\{0,1\}^{\ell}$.  Form the multigraphs $G_R$ and $H_R$ by sequentially deleting when $r_i=0$, or contracting when $r_i=1$, each edge $e_i$ in turn.  Repeated application of the deletion-contraction formula (\ref{def}) to $T(G;x,y)$ and $T(H;x,y)$ then produces 
\begin{equation} \label{decom1}
T(H;x,y)-T(G;x,y) = \sum_{R \in \mathcal{R}} (T(H_R;x,y) -  T(G_R;x,y))
\end{equation}
where $\mathcal{R} \subseteq \{0,1\}^{\ell}$ denotes the permissible sequences of deletions and contractions on $S$.  (Bridges and loops may be formed by certain sequences of deletions and contractions, which means not all possible binary vectors may appear.)  By the above it suffices to show that $G_R \preccurlyeq H_R$ for every $R \in \mathcal{R}$.

Let $R \in \mathcal{R}$ be fixed, then, and consider $G_R$ and $H_R$.  Since every edge of $S$ has been either deleted or contracted, the graph $G_R$ is a multistar, that is, the blocks $M_{m_1}, \dots, M_{m_k}$ of $G_R$ are each multiedge blocks as in (\ref{mte}), with $m_1, \dots, m_k$ denoting the edge multiplicities, and the cutvertex $v$ is the common vertex of the multiedge blocks.  
Note that, since $G$ consists of two blocks, then $G_R$ contains at least two multiedge blocks and $k \ge 2$.  
As in $G$ and $H$ themselves, the corresponding graph $H_R$ differs from $G_R$ only in one edge.  In $G_R$ we may assume that the edge is part of the first block $M_{m_1}$, while in $H_R$ the edge connects the non-$v$ vertices of block $M_{m_1}$ and $M_{m_2}$.  Again abusing notation somewhat, we let $e$ denote that edge in both $G_R$ and $H_R$.   Since $G_R - e = H_R - e$, by (\ref{def}) we thus have
\[T(H_R;x,y) - T(G_R;x,y) = T(H_R/e;x,y) - T(G_R/e;x,y)\]
  
Thus it suffices to show that $G_R/e \preccurlyeq H_R/e$.  Since in $G_R$ we have $e \in E(M_{m_1})$, then $G_R/e = (m_1-1)C_1 \cdot M_{m_2} \cdot \ldots \cdot M_{m_k}$ with $C_1$ denoting a loop. 
In $H_R$ edge $e$ instead connects the non-$v$ vertex of the first multiedge block to the non-$v$ vertex of $M_{m_2}$, and thus 
the graph $H_R/e$ combines these blocks into the single multiedge block $M_{m_1+m_2-1}$.  Therefore, using (\ref{fact}) we have
\begin{align*} 
T(H_R/e;x,y) &- T(G_R/e;x,y) \\
&= T(M_{m_1+m_2-1} \cdot M_3 \cdot \ldots \cdot M_{m_k};x,y)-  T((m_1-1)C_1\cdot M_{m_2} \cdot M_{m_3} \cdot \ldots \cdot M_{m_k};x,y) \\
&= \left(T(M_{m_1+m_2-1};x,y) - T((m_1-1)C_1 \cdot M_{m_2};x,y)\right)\prod_{i=3}^k T(M_{m_i};x,y) 
\intertext{When $m_1 = 1$ then the above expression reduces to zero (in fact $G_R/e$ and $H_R/e$ are block-isomorphic) and we have trivially $G_R/e \preccurlyeq H_R/e$.  Otherwise, using (\ref{mte}), and then Lemma \ref{tlem1} with $a=m_1$ and $b=m_2$, the expression reduces to}
&= \left(\left(x + y+ \dots + y^{m_1+m_2-2}\right) - y^{m_1-1}\left(x + y + \dots + y^{m_2-1} \right)\right)\prod_{i=3}^k T(M_{m_i};x,y) \\
&= (x+y-xy)\left(1 + y + \dots + y^{m_1-2}\right) \prod_{i=3}^k T(M_{m_i};x,y) 
\end{align*}
where it is understood that if $m_1 = 1$ then the second parenthetical expression in the final line is equal to 0.  Clearly $(1+ y + \dots + y^{m_2-2}) \prod_{i=3}^k T(M_{m_i};x,y)$ is a polynomial with non-negative coefficients, and we have $G_R/e \preccurlyeq H_R/e$, which suffices to show that $G \preccurlyeq H$ as desired. The proof is completed by noting that, when $G_1 \neq K_2$, then $G_1$ is necessarily 2-connected, which implies that a cycle including $u,v$ exists.  At least one of the $R \in \mathcal{R}$ contracts every edge of this cycle that is not incident to $v$, and this results in $m_1 \ge 2$ in $G_R$ for that $R$.  Thus $T_H(x,y) - T_G(x,y) =(x+y-xy)P(x,y)$ where $P(x,y)$ is non-zero, giving $G \prec H$.
\end{proof}

As an immediate consequence we have the Tutte polynomial poset structure for the unicyclic graph class $\mathcal{G}_{n,n}$, by repeatedly incorporating bridges successively into the cycle, i.e., by applying Theorem \ref{comb} with $G_1$ equal to the cycle and $G_2$ equal to a bridge incident to the cycle.

\begin{thm}  \label{gnn}
The Tutte polynomial poset for $G_{n,n}$ is the chain
\[C_3 \cdot (n-3)K_2 \prec C_4 \cdot (n-4)K_2 \prec \ldots \prec C_{n-1} \cdot K_2 \prec C_n.\]
In particular, the unique Tutte-maximum graph of the poset is the cycle $C_n$, and the minimum graphs of the poset have blocks consisting of a single triangle $C_3$ and $(n-3)$ bridges.
\end{thm}

Before we move on we note one particular instance of the Theorem \ref{comb} which will be useful in the other edge-sparse graph classes as well, which essentially says that contracting a bridge and ``replacing'' it by subdividing an edge in a 2-connected block (thus keeping the same number of vertices and edges in the whole graph) will result in a better graph in the Tutte poset.

\begin{cor} \label{bridgeelim}
Let $G$ be a 2-connected graph.  Then $G \cdot K_2 \prec G'$, where $G'$ is $G$ with any edge subdivided. 
\end{cor}  

\begin{proof}
Let $uv$ be any edge of $G$.  Since the Tutte polynomial is block-invariant, we may assume that the bridge is incident with $v$, and call its non-$v$ vertex $w$.  Applying Theorem \ref{comb} with $G_1 = G$ and $G_2 = K_2$ now gives the result.
\end{proof}

\section{Tutte polynomial posets for $\mathcal{G}_{n,n+1}$}

An important family of graphs for this section and the next one are the generalized theta graphs.  
A \textit{generalized theta graph} $\theta(a_1,\dots,a_k)$ is the graph consisting of two vertices connected by $k$ vertex-independent paths, each $a_i$ denoting the length in edges of the $i^{th}$ path.  A generalized theta graph with $k=3$ is often called a \textit{theta graph} $\theta(a,b,c)$. (Theta graphs look something like the letter $\theta$).  

Theta graphs are the only 2-connected graphs in $\mathcal{G}_{n,n+1}$ \cite{bbst85}.  This fact effectively makes them the only candidates for maximal graphs in the class, which we note in the next theorem.

\begin{lem} \label{thet}
For every $G \in \mathcal{G}_{n,n+1}$ there is a theta graph $H$ such that $G \preccurlyeq H$.
\end{lem}

\begin{proof}
Let $H$ be a maximal graph in the Tutte poset.  By Corollary \ref{bridgeelim} no bridges are present in $H$.
Thus, given $H$ has $n+1$ edges, $H$ must consist of either a theta graph or two cycles sharing a cutvertex.  However if two cycles appear as blocks in $H$, say $C_a \cdot C_b$, then applying Theorem \ref{comb} with say, $G_1=C_a$ and $G_2=C_b$ we obtain the theta graph $\theta(a,1,b-1)$ which is larger in the poset, contradicting maximality.
\end{proof} 

With regard to maximum graphs in this class then it is only a question of which theta graph is maximum, or if there might exist multiple maximal theta graphs.  Note that generalized theta graphs may be thought of as consisting of \textit{parallel ears}, that is, ears that share common endvertices.  The next theorem then, which shows that making parallel ears ``more equal'' in length will produce a better graph in the Tutte poset, will help settle that issue.

\begin{thm} \label{par}
Let $G$ be a graph containing two parallel ears $E_a, E_b$ with common endvertices $x,y$, with $a-1 > b$.  Let $H$ be the same graph but with the parallel ears $E_a,E_b$ replaced with parallel ears $E_{a-1},E_{b+1}$.  Then $G \prec H$.
\end{thm}

\begin{proof}
Let $e \in E(G)$ be any edge of the ear $E_a$ and let $f \in E(H)$ be any edge of the ear $E_{b+1}$.  Applying (\ref{def}) to these edges, and noting that $G/e = H/f$, we obtain 
\begin{align*}
T(H;x,y) - T(G;x,y) &= T(H-f;x,y) - T(G-e;x,y) \\
&= T((H-E_{b+1})\cdot bK_2;x,y) - T((G-E_a)\cdot (a-1)K_2;x,y).
\end{align*}
We now show $(G-E_a)\cdot (a-1)K_2 \prec (H-E_{b+1})\cdot bK_2$.
 Note that in the graph $G-E_a$ the ear $E_b$ is still present, and in $H-E_{b+1}$ the ear $E_{a-1}$ is still present, and that $(G-E_a)-E_b = (H-E_{a-1})-E_{b+1}$.   To emphasize this, we rewrite $G-E_a = G':E_b$ and $H-E_{b+1}=G':E_{a-1}$ where $G' =(G-E_a)-E_b = (H-E_{a-1})-E_{b+1}$ and $G':E$ indicates the graph $G'$ with ear $E$ connecting vertices $x,y$.  (The `colon' is intended to suggest the vertices $x,y$ of $G'$.)  

Now since $a-1>b$, the ear $E_b$ in $G':E_b$ is shorter than the ear $E_{a-1}$ in $G':E_{a-1}$, and there are more bridges in $G':E_b \cup (a-1)K_2$ than there are in $G':E_{a-1}\cup bK_2$.  Thus by repeatedly applying the operation of Corollary \ref{bridgeelim} to the graph $G':E_b \cup (a-1)K_2$ we obtain the chain of inequalities
\[(G':E_{b}) \cdot (a-1)K_2 \prec (G':E_{b+1}) \cdot (a-2)K_2\prec \dots \prec (G':E_{a-1}) \cdot bK_2\]
where $G':E_{b+1}$ indicates that the ear $E_b$ in $G':E_b$ has been lengthened by one edge. 
But $(G':E_{a-1}) \cdot bK_2 = (H-E_{b+1}) \cdot bK_2$, and thus  $(G-E_a)\cdot (a-1)K_2 \prec (H-E_{b+1}) \cdot b K_2$ as required. 
\end{proof}

The last theorem, together with Lemma \ref{thet}, immediately implies the existence of maximum theta graphs, and hence maximum graphs, in this class.

\begin{thm}  \label{thetamax}
The unique Tutte-maximum graph in the $(n,n+1)$ Tutte polynomial poset is the theta graph with path lengths as equal as possible.
\end{thm} 

We make note here of a variation on Theorem \ref{par} that will be useful in the next section.  Letting $x=y$ in that theorem we have the following, which says essentially that making the lengths of two cycle blocks ``more equal'' produces a better graph.  (This is also apparent via Theorem \ref{cyc} and some tedious algebra, but is immediate via Theorem \ref{par}.)

\begin{thm} \label{cycv}
Let $C_a \cdot C_b$ be such that $a- 1>b$.  Then $C_a \cdot C_b \prec C_{a-1}\cdot C_{b+1}$.
\end{thm}

Finally, we note the increasing complexity of the Tutte polynomial posets in $\mathcal{G}_{n,n+1}$ as compared to $\mathcal{G}_{n,n}$.  Unlike in the previous $(n,n)$ class the $(n,n+1)$ Tutte polynomial posets are typically not chains.  The smallest non-chain poset (obtained using \textit{Mathematica}) is the $(6,7)$ Tutte polynomial poset which appears in Figure \ref{poset67} at the end of the paper.  In that poset the graph $\theta(2,2,2) \cdot K_2$ is incomparable with both the graphs $\theta(1,2,4)$ and $C_3 \cdot C_4$.  Based on our computations, the Tutte posets in this $\mathcal{G}_{n,n+1}$ class appear to rapidly increase in complexity as $n$ increases.

\section{Tutte polynomial posets for $\mathcal{G}_{n,n+2}$}

As in the previous class, it is important to first identify which 2-connected graphs are possible in this class.  These were determined in \cite{ls24}, where it was shown there are four general families of 2-connected graphs here:  the generalized theta graphs with four ears $\theta(a,b,c,d)$, and graph families which we will call the delta graphs $\Delta(a,b,c,d,e)$, the box graphs $B(a,b,c,d,e,f)$, and the cylinder graphs $C(a,b,c,d,e,f)$.  Generalized theta graphs have been defined in the previous section.  The form of the delta graph, box graph, and cylinder graph families are pictured in Figure \ref{fig2}.    The formal definitions of the these families are as follows.  We say that a graph is a \textit{delta graph} $\Delta(a,b,c,d,e)$ if it has three vertices, say $x,y,z$, with one ear of length $e$ between the pair of vertices $(x,y)$, and two parallel ears of lengths $a,b$ (resp. $c,d$) in between the vertex pair $(x,z)$ (resp. $(y,z)$).  We say a graph is a \textit{box graph} $B(a,b,c,d,e,f)$ if it has four vertices, say $u,v,x,y$, with single ears of lengths $a,b,c,d,e,f$ between vertex pairs $(u,v)$, $(x,y)$, $(u,x)$, $(v,y)$, $(v,x)$, and $(u,y)$ respectively.  Finally, we say that a graph is a \textit{cylinder graph} $C(a,b,c,d,e,f)$ if it has four vertices, say $u,v,x,y$, with single ears of length $e,f$ between vertex pairs $(u,v)$ and $(x,y)$, and two parallel ears of lengths $a,b$ and $c,d$ respectively between vertex pairs $(u,x)$ and $(v,y)$.

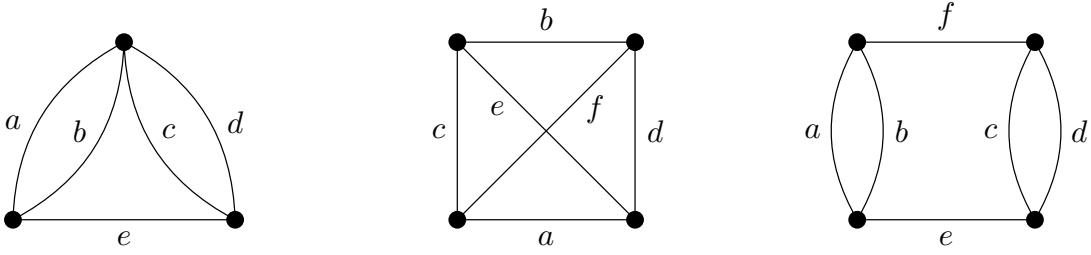
\begin{figure} 
\begin{center}
\resizebox{15cm}{!}{
\begin{tikzpicture}

\fill (0,0) node[black,right] {} circle (0.1cm);
\fill (2.5,0) node[black,right] {} circle (0.1cm);
\fill (1.25,2) node[black,right] {} circle (0.1cm);

\draw (0,0) -- (2.5,0)  node [midway, below] {\small $e$};
\draw (0,0) to [bend right=30] (1.25,2); 
\path (0,0) edge [bend left] (1.25,2); 
\path (2.5,0) edge [bend right] (1.25,2); 
\path (2.5,0) edge [bend left] (1.25,2); 
\node at (0,1.1) {\small $a$};
\node at (2.5,1.1) {\small $d$};
\node at (0.75,1) {\small $b$};
\node at (1.75,1) {\small $c$};

\begin{scope}[xshift=5cm]
\fill (0,0) node[black,right] {} circle (0.1cm);
\fill (2,0) node[black,right] {} circle (0.1cm);
\fill (2,2) node[black,right] {} circle (0.1cm);
\fill (0,2) node[black,right] {} circle (0.1cm);

\draw (0,0) -- (2,0)  node [midway, below] {\small $a$};
\draw (2,0) -- (2,2)  node [midway, right] {\small $d$};
\draw (2,2) -- (0,2)  node [midway, above] {\small $b$};
\draw (0,0) -- (0,2)  node [midway, left] {\small $c$};
\draw (0,0) -- (2,2);
\draw (2,0) -- (0,2);
\node at (0.45,1.25) {\small $e$};
\node at (1.55,1.25) {\small $f$};

\end{scope}

\begin{scope}[xshift=9.5cm]

\fill (0,0) node[black,right] {} circle (0.1cm);
\fill (2,0) node[black,right] {} circle (0.1cm);
\fill (2,2) node[black,right] {} circle (0.1cm);
\fill (0,2) node[black,right] {} circle (0.1cm);

\draw (0,0) -- (2,0)  node [midway, below] {\small $e$};
\draw (0,2) -- (2,2)  node [midway, above] {\small $f$};
\draw (0,0) to [bend right=30] (0,2); 
\draw (0,0) to [bend left=30] (0,2); 
\draw (2,0) to [bend right=30] (2,2); 
\draw (2,0) to [bend left=30] (2,2); 
\node at (-0.5,1) {\small $a$};
\node at (1.5,1) {\small $c$};
\node at (0.5,1) {\small $b$};
\node at (2.5,1) {\small $d$};

\end{scope}

  \end{tikzpicture}
}
\end{center}
\caption{Delta graphs $\Delta(a,b,c,d,e)$, box graphs $B(a,b,c,d,e,f)$, and cylinder graphs $C(a,b,c,d,e,f)$.  Each labeled line indicates a path with length in edges equal to the label.} 
\label{fig2}
\end{figure}

We now show that the box graphs  play a similar role in $\mathcal{G}_{n,n+2}$ as the theta graphs did in $\mathcal{G}_{n,n+1}$ in that, if $G \in \mathcal{G}_{n,n+2}$ is any other type of graph then there will be a box graph $H$ such that $G \prec H$.  From Theorem \ref{comb} we know that any maximal graphs in the Tutte poset must be 2-connected, so to prove that box graphs are the maximal graphs here it suffices to show that for each generalized theta graph $\theta(a,b,c,d)$, each delta graph $\Delta(a,b,c,d,e)$, and each cylinder graph $C(a,b,c,d,e,f)$ there is a box graph $B(a,b,c,d,e,f)$ that is larger in the poset.  That is accomplished in the next three theorems.

\begin{thm}
Let $G$ be a generalized theta graph $\theta(a,b,c,d)$ with $a = \min\{a,b,c,d\}$.  Let $H$ be the box graph $B(1,1,b-1,c-1,a,d)$.  Then $G \prec H$.
\end{thm}

\begin{proof}
We note that, since $G$ is assumed to be simple and $a$ is the minimum path length in $G$ then $b,c \ge 2$ and $H$ is well-defined.  Now, letting $E_a$ denote the ear of length $a$ in both $G$ and $H$, by construction we have $G-E_a = H-E_a = \theta(b,c,d)$.  Therefore, applying Theorem \ref{e1} to the ear $E_a$ in both $G$ and $H$, we obtain
\begin{align*} 
T(H;x,y) - T(G;x,y) &= T(H/E_a;x,y) - T(G/E_a;x,y) \\
&= T(\Delta(b-1,1,1,c-1,d);x,y) -T(C_b \cdot C_c \cdot C_d;x,y)
\end{align*}
Hence it suffices to show that $C_b \cdot C_c \cdot C_d\prec \Delta(b-1,1,1,c-1,d)$.  Now in $C_b \cdot C_c \cdot C_d$ let $e$ be an edge of $C_d$ and $f$ be an edge of $C_b$.  Applying Theorem \ref{comb} with $G_1=C_d$, $G_2 = C_b$, with $uv=e$, $vw=f$ we obtain
\[C_b \cdot C_c \cdot C_d \prec \theta(b-1,1,d) \cdot C_c.\]   
Now in $\theta(b-1,1,d) \cdot C_c$ let $e$ be an edge of the ear $E_d$ in $\theta(b-1,1,d)$ and $f$ an edge of $C_c$.  Now applying Theorem \ref{comb} again with $G_1 = \theta(b-1,1,d), G_2 = C_c$, with $uv=e$ and $vw = f$ we obtain
\[\theta(b-1,1,d) \cdot C_c \prec \Delta(b-1,1,1,c-1,d).\]
Thus $C_b \cdot C_c \cdot C_d\prec\Delta(b-1,1,1,c-1,d)$, and so $G \prec H$ as desired.
\end{proof}

\begin{thm} \label{del}
Let $G$ be the delta graph $\Delta(a,b,c,d,e)$ with $d = \max\{a,b,c,d\}$.  Let $H$ be the box graph $B(1,e,a,c,b,d-1)$.  Then $G \prec H$.
\end{thm}

\begin{proof}
Note that since $G$ is assumed to be simple then $d \ge 2$ and $H$ is well-defined.  Now, letting $E_a$ denote the ear of length $a$ in both $G$ and $H$, by construction we have $G-E_a = H-E_a = \theta(b+e,c,d)$.  Hence, applying Theorem \ref{e1} to the ear $E_a$ in both $G$ and $H$, we obtain
\begin{align*} 
T(H;x,y) - T(G;x,y) &= T(H/E_a;x,y) - T(G/E_a;x,y) \\
&= T(\Delta(b,1,e,d-1,c);x,y) -T(C_b \cdot \theta(c,d,e);x,y)
\end{align*}
Now let $E_c$ denote the ear of length $c$ in both $\theta(c,d,e)$ and $\Delta(b,1,e,d-1,c)$.  Applying Theorem \ref{e1} now to the ear $E_c$ in both $\theta(c,d,e)$ and $\Delta(b,1,e,d-1,c)$ the last expression above then becomes
\begin{align*}
((1+x&+x^2+\dots+x^{c-1})T(\Delta(b,1,e,d-1,c)-E_c;x,y)+T(\Delta(b,1,e,d-1,c)/E_c;x,y)) \\&- \left((1+x+x^2+\dots+x^{c-1})T(C_b \cdot \theta(c,d,e)-E_c;x,y)+T(C_b \cdot \theta(c,d,e)/E_c;x,y)\right)\\
&= (1+x+x^2+\dots+x^{c-1})T(C_{b+1} \cdot C_{d+e-1};x,y)+ T(\theta(b,1,e,d-1);x,y)\\&-(1+x+x^2+\dots+x^{c-1})T(C_b \cdot C_{d+e};x,y)-T(C_b \cdot C_d \cdot C_e;x,y) 
\end{align*}
Hence it suffices to show that $C_b \cdot C_{d+e} \prec C_{b+1} \cdot C_{d+e-1}$ and $C_b \cdot C_d \cdot C_e \prec \theta(b,1,e,d-1)$.  However, $C_b \cdot C_{d+e} \prec C_{b+1} \cdot C_{d+e-1}$ follows directly from Theorem \ref{cycv}, so to complete the proof it only remains to show that $C_b \cdot C_d \cdot C_e \prec \theta(b,1,e,d-1)$.

To see that $C_b \cdot C_d \cdot C_e \prec \theta(b,1,e,d-1)$, we apply Theorem \ref{comb} twice.  First, let $G_1 = C_b$, $G_2 = C_d$ with $uv$ any edge of $C_b$ and $vw$ any edge of $C_d$.  By Theorem \ref{comb} we have $C_b \cdot C_d \cdot C_e \prec \theta(b,1,d-1) \cdot C_e$.  Now take $G_1 = C_e$, $G_2 = \theta(b,1,d-1)$, and $uv$ to be any edge of $C_e$ and $vw$ to be the edge of $\theta(b,1,d-1)$ that comprises the path of length 1.  By Theorem \ref{comb} we arrive at $\theta(b,1,d-1) \cdot C_e \prec \theta(b,1,e,d-1)$.  Therefore $C_b \cdot C_d \cdot C_e \prec \theta(b,1,e,d-1)$ as required.
\end{proof}

\begin{thm}
Let $G$ be the cylinder graph $C(a,b,c,d,e,f)$ with $d = \max\{a,b,c,d\}$.  Let $H$ be the box graph $B(1,e+f,a,c,b,d-1)$.  Then $G \prec H$.
\end{thm}

\begin{proof}
Letting $E_c$ denote the ear of length $c$ in both $G$ and in the delta graph $D=\Delta(a,b,c,d,e+f)$, and note that $G-E_c = D- E_c$ and that $G / E_c = \theta(a,b,e+f) \cdot C_d = D / E_c$. Thus $T(G;x,y) = T(D;x,y)$, and the result now follows from the previous theorem.
\end{proof}

With regard to maximum graphs in this class then it is only a question of which box graph is maximum, or if there might exist multiple maximal box graphs.  The following theorems, then, which describe how adjacent ears may be ``evened out'' in length, will help settle the maximum question.

\begin{thm} \label{boxev}
Let $G=B(a,b,c,d,e,f)$ be a box graph such that $b + f - 1 > a + e$.  Let $H = B(a+1,b,c,d,e,f-1)$.  Then $G \prec H$.
\end{thm}

\begin{proof}
Note that if $b+f-1 > a+e$ then we may assume that $f>1$ and thus $H$ is well-defined.  Now, letting $E_c$ denote the ear of length $c$ in both $G$ and $H$, by construction we have $G - E_c = H-E_c = \theta(d,b+e,a+f)$.  
Hence, applying Theorem \ref{e1} to the ears $E_c$ we have
\begin{align*}
T(H;x,y) - T(G;x,y) &= T(H/E_c;x,y) - T(G/E_c;x,y)\\
&= T(\Delta(a+1,e,b,f-1,d);x,y) - T(\Delta(a,e,b,f,d);x,y)
\end{align*}
 Hence it suffices to show that $\Delta(a,e,b,f,d) \prec \Delta(a+1,e,b,f-1,d)$.  Now, letting $E_d$ denote the ear of length $d$ in both of these graphs, applying Theorem \ref{e1} the above is equal to
\begin{align*}
(1&+x +\dots + x^{d-1})T(\Delta(a+1,e,b,f-1,d)-E_d;x,y) + T(\Delta(a+1,e,b,f-1,d)/E_d;x,y) \\
&- (1+x+\dots+x^{d-1})T(\Delta(a,e,b,f,d)-E_d;x,y) - T(\Delta(a,e,b,f,d)/E_d;x,y)\\
&= (1+x+\dots +x^{d-1})T(C_{a+e+1}\cdot C_{b+f-1};x,y) + T(\theta(a+1,e,b,f-1);x,y)\\
& \qquad -(1+x+\dots +x^{d-1})T(C_{a+e} \cdot C_{b+f};x,y) - T(\theta(a,e,b,f);x,y)
\end{align*}
Hence it suffices to show that $C_{a+e}\cdot C_{b+f} \prec C_{a+e+1}\cdot C_{b+f-1}$ and $\theta(a,b,e,f) \prec \theta(a+1,b,e,f-1)$.  But since $b+f-1 > a+e$ both these statements follow from previous results, the first by Theorem \ref{cycv}
and the second from Theorem \ref{par}, completing the proof.
\end{proof}

\begin{thm} \label{boxev2}
Let $G=B(a,b,c,d,e,f)$ be a box graph such that $b-1 > e$ and $a > f$.  Let $H = B(a-1,b-1,c,d,e+1,f+1)$.  Then $G \prec H$.
\end{thm}

\begin{proof}
Note that since $b-1 > e$ and $a>f$ we may assume that both $a,b>1$ and thus $H$ is well-defined.  As in the previous proof, by construction we have $G - E_c = H-E_c = \theta(d,b+e,a+f)$, and hence it suffices to show that $\Delta(a,e,b,f,d) \prec \Delta(a+1,e+1,b-1,f+1,d)$.  In this case, however, we also have that $\Delta(a,e,b,f,d) - E_d  = \Delta(a-1,e+1,b-1,f+1,d) - E_d= C_{a+e} \cdot C_{b+f}$.  Hence after applying Theorem \ref{e1} we obtain
\begin{align*}
T(H;x,y) - T(G;x,y) &=  T(\Delta(a-1,e+1,b-1,f+1,d)/E_d;x,y) - T(\Delta(a,e,b,f,d)/E_d;x,y) \\
&= T(\theta(a-1,e+1,b-1,f+1);x,y) - T(\theta(a,e,b,f);x,y) 
\end{align*}
Hence it suffices to show that $\theta(a,e,b,f) \prec \theta(a-1,e+1,b-1,f+1)$.  But since $b-1 > e$ and $a > f$ this follows from Theorem \ref{par}, completing the proof.
\end{proof}

We can now identify the unique maximum graph in the graph class $\mathcal{G}_{n,n+2}$.

\begin{thm} \label{plus2}
The unique Tutte-maximum graph in the $(n,n+2)$ Tutte polynomial poset is the box graph $B(a,b,c,d,e,f)$ with ear lengths as evenly distributed as possible and, furthermore, with as many of the pairs $(E_a,E_b)$, $(E_c,E_d)$, $(E_e,E_f)$, equal in length as possible. 
\end{thm}

To clarify the statement on pairs of ears, another way to say this is that the maximum graphs for graph classes $\mathcal{G}_{4,6}$, $\mathcal{G}_{5,7}$, $\mathcal{G}_{6,8}$, etc., are $B(1,1,1,1,1,1)$, $B(2,1,1,1,1,1)$, $B(2,2,1,1,1,1)$, $\dots$, $B(2,2,2,2,2,2)$, $B(3,2,2,2,2,2)$, $B(3,3,2,2,2,2)$, and so on.  In the proof below we will call these graphs $B_n^*$.

\begin{proof}
Assume to the contrary that there is a graph $G \in \mathcal{G}_{n,n+2}$ such that $B_n^* \preccurlyeq G$, and take $G$ to be such a graph that is maximal in the poset.  By the previous results of this section we may assume that $G$ is a box graph $B(a,b,c,d,e,f)$.  Now examine two independent ears of $G$, that is, ears that have no common vertices, such as $E_e, E_f$.  In independent ears $E_e,E_f$, if $f-1 > e$, then by taking $E_b$ to be the longer ear of $E_a,E_b$ in the cycle formed by $E_a,E_b,E_c,E_f$ we have $b+f-1 > a+e$, which by Theorem \ref{boxev} means that $G$ is not maximal, a contradiction.  Hence every two independent ears of $G$ must differ in lengths by at most one.

Next examine two adjacent ears in $G$, that is, ears that share one common endvertex, such as $E_b,E_e$.  Assume in adjacent ears $E_b,E_e$ we have $b-1 > e$.  Now either $f \ge a$, in which case  we have $b+f-1>a+e$ and $G$ is not maximal by Theorem \ref{boxev}, or else  $a >f$, in which case $G$ is not maximal by Theorem \ref{boxev2}.  Since there are no parallel ears in $G$, we see in fact that every pair of ears of $G$ must differ in lengths by at most one, and thus the ear lengths are as equal as possible.

Finally, consider the case where the ear lengths of $G$ are as equal as possible, but fewer than possible of the pairs $(E_a,E_b)$, $(E_c,E_d)$, $(E_e,E_f)$, are equal in length.  Then there exists two pairs, say $(E_a,E_b)$ and $(E_e,E_f)$, that are unequal in length, which implies there exists the cycle $E_a,E_b,E_e,E_f$, such that $a=e=t$ and $b=f=t+1$ for some $t \ge 1$.  But this implies $b+f-1>a+e$, again a contradiction.  Thus as many of the pairs of ears must be equal as possible, and in fact $G = B_n^*$.  
\end{proof}

\section{Conclusion}

The cycles $C_n$ in $G_{n,n}$, the theta graphs with path lengths as even distributed as possible in $G_{n,n+1}$, and the box graphs of Theorem \ref{plus2} are the UOR graphs  for their classes as identified in \cite{bls91}.  Thus Theorems \ref{gnn}, \ref{thetamax}, and \ref{plus2} generalize \cite{bls91} by showing that those graphs are maximum graphs in their respective classes not just for all-terminal reliability, but for a wide variety of graph parameters associated with the Tutte polynomial, including the number of spanning trees, spanning forests, and spanning connected subgraphs;  the number of various types of orientations of $G$, including acyclic orientations, totally cyclic orientation, and acyclic orientations with a single source; and various types of partial orientations and fourientations of $G$.  In addition, the cycles, theta graphs and box graphs described have the maximum coefficients (in absolute value) in their classes for the chromatic polynomial and flow polynomial; generating functions for spanning forests of $i$ components, and spanning subgraphs of $i$ edges; generating functions relating to orientations, partial orientations, and fourientations of a graph; and the generating function for the number of critical configurations of level $i$ of the Abelian sandpile model \cite{m97}.  

While the graph classes studied here all have unique Tutte-maximum graphs, our computation of small Tutte posets shows that this is not always the case.  The smallest such class, which is also the smallest class with no UOR graph \cite{mcpp91}, is $(6,11)$, which has two maximal graphs in its Tutte poset.  (The two graphs have $2P_3$ and $P_4 \cup P_2$ as their complements, respectively.)  Thus the $(n,n+5)$ Tutte polynomial posets do not always have maximum graphs, at least not for all $n$.  This is also the case for the $(n,n+4)$ Tutte polynomial posets, as illustrated by the $(7,11)$ Tutte polynomial poset described in the Introduction.  However, in the class $\mathcal{G}_{n,n+3}$ we conjecture that there is a unique maximum graph.  Just as the unique maximum for $\mathcal{G}_{n,n+2}$ was a particular ``equally distributed'' subdivision of $K_4$, a particular ``equally distributed'' subdivision of $K_{3,3}$ was shown to be maximum for all-terminal reliability by Wang \cite{w94}.  Some of the graphs in \cite{w94} for particular $n$ were recently shown to be incorrect in \cite{ls24}, which provided the correct graphs for those cases.
We conjecture that the graphs given in \cite{w94} and \cite{ls24} which maximize all-terminal reliability are more generally maximum graphs in the $(n,n+3)$ Tutte polynomial poset.  If true then---like the maximum graphs found here---these graphs are in fact maximum in their class for far more parameters than just all-terminal reliability.  We believe some of the tools in this paper may prove useful in determining whether our conjecture is correct or not.  

Finally, as far as larger, more edge-dense graph classes, other Tutte polynomial parameters may be of some guide here.  Much work, for example, has been done on determining which graphs are \emph{t-optimal}, in other words, maximize the number of spanning trees.  We mention one fact in particular.  In \cite{c81} it was shown that regular complete multipartite graphs maximize the number of spanning trees in the classes $\gnm$ where they exist.  We conjecture that, in these graph classes, the regular complete multipartite graphs are in fact Tutte-maximum.

\section*{Acknowledgements}

The authors would like to acknowledge Lorents Landgren and Jeffrey Stief for their helpful comments and for providing the preprint \cite{ls24}, and also Pablo Romero, who discovered a small gap in the original proof of Theorem \ref{plus2}, and provided the argument of Theorem \ref{boxev2} that fixed that gap.  Many thanks to all.

\begin{figure}[h]  
\begin{center}
\resizebox{0.6\linewidth}{!}{
\begin{tikzpicture}

\tikzset{mynode/.style={draw,circle,fill=black,inner sep=10pt}
    }

\tikzset{
   G5/.pic={
 \node [mynode] (A2) at (180:6cm) {};
	\node [mynode] (B2) at (90:5cm) {};
 	\node [mynode] (C2) at (0:6cm) {};
	\node [mynode] (D2) at (-90:5cm) {};
	\node [mynode] (E2) at (0:2cm) {};
	\node [mynode] (F2) at (180:2cm) {};
	\draw[line width=2mm] (A2) -- (B2) -- (C2) -- (D2) -- (A2);
	\draw[line width=2mm] (A2) -- (F2) -- (E2) -- (C2);
	}
}

\tikzset{
   G7/.pic={
 \node [mynode] (A7) at (120:5cm) {};
	\node [mynode] (B7) at (60:5cm) {};
 	\node [mynode] (C7) at (0:5cm) {};
	\node [mynode] (D7) at (-60:5cm) {};
	\node [mynode] (E7) at (-120:5cm) {};
	\node [mynode] (F7) at (180:5cm) {};
	\draw[line width=2mm] (A7) -- (B7) -- (C7) -- (D7) -- (E7) -- (F7) -- (A7);
	\draw[line width=2mm] (F7) -- (C7);
	}
}

\tikzset{
   G6/.pic={\begin{scope}[rotate=90]
 \node [mynode] (A6) at (180:5cm) {};
	\node [mynode] (B6) at (120:5cm) {};
 	\node [mynode] (C6) at (60:5cm) {};
	\node [mynode] (D6) at (0:5cm) {};
	\node [mynode] (E6) at (-60:5cm) {};
	\node [mynode] (F6) at (-120:5cm) {};
	\draw[line width=2mm] (A6) -- (B6) -- (C6) -- (D6) -- (E6) -- (F6) -- (A6);
	\draw[line width=2mm] (B6) -- (F6);
	\end{scope}
		}
}

\tikzset{
   G8/.pic={
 \node [mynode] (A8) at (180:7cm) {};
	\node [mynode] (B8) at (135:5cm) {};
 	\node [mynode] (C8) at (0:0cm) {};
	\node [mynode] (D8) at (45:5cm) {};
	\node [mynode] (E8) at (-45:5cm) {};
	\node [mynode] (F8) at (-135:5cm) {};
	\draw[line width=2mm] (A8) -- (B8) -- (C8) -- (D8) -- (E8) -- (C8) --  (F8) -- (A8);
	}
}

\tikzset{
   G4/.pic={
 \node [mynode] (A4) at (54:5cm) {};
	\node [mynode] (B4) at (126:5cm) {};
 	\node [mynode] (C4) at (198:5cm) {};
	\node [mynode] (D4) at (270:5cm) {};
	\node [mynode] (E4) at (-18:5cm) {};
	\node [mynode] (F4) at (-9:9.8cm) {};
	\draw[line width=2mm] (A4) -- (B4) -- (C4) -- (D4) -- (E4) -- (A4);
	\draw[line width=2mm] (C4) -- (E4);
		\draw[line width=2mm] (F4) -- (E4);
	}
}

\tikzset{
   G2/.pic={\begin{scope}[rotate=90]
 \node [mynode] (A2) at (180:5cm) {};
	\node [mynode] (B2) at (90:5cm) {};
 	\node [mynode] (C2) at (0:5cm) {};
	\node [mynode] (D2) at (-90:5cm) {};
	\node [mynode] (E2) at (0:0cm) {};
	\node [mynode] (F2) at (-90:10cm) {};
	\draw[line width=2mm] (A2) -- (B2) -- (C2) -- (D2) -- (A2);
	\draw[line width=2mm] (B2) -- (E2) -- (D2) -- (F2);
	\end{scope}
	}
}

\tikzset{
   G3/.pic={
 \node [mynode] (A3) at (180:5cm) {};
	\node [mynode] (B3) at (0:0cm) {};
 	\node [mynode] (C3) at (0:5cm) {};
	\node [mynode] (D3) at (-60:5cm) {};
	\node [mynode] (E3) at (-120:5cm) {};
	\node [mynode] (F3) at (90:5cm) {};
	\draw[line width=2mm] (A3) -- (B3) -- (C3) -- (D3) -- (B3) -- (E3) -- (A3);
	\draw[line width=2mm] (B3) -- (F3);
	}
}

\tikzset{
   G1/.pic={
 \node [mynode] (A1) at (180:5cm) {};
	\node [mynode] (B1) at (90:5cm) {};
 	\node [mynode] (C1) at (0:5cm) {};
	\node [mynode] (D1) at (-90:5cm) {};
	\node [mynode] (E1) at (180:10cm) {};
	\node [mynode] (F1) at (0:10cm) {};
	\draw[line width=2mm] (A1) -- (B1) -- (C1) -- (D1) -- (A1);
	\draw[line width=2mm] (E1) -- (A1);
	\draw[line width=2mm] (C1) -- (F1);
	\draw[line width=2mm] (A1) -- (C1);
	}
}

\begin{scope}[scale=1.5]
\pic at (0,36) {G5};
\pic at (0,24) {G7};
\pic at (20,11) {G6};
\pic at (20,-1) {G8};
\pic at (0,-18) {G4};
\pic at (-20,6) {G2};
\pic at (0,-32) {G3};
\pic at (0,-46) {G1};
\end{scope}

\begin{scope}
\draw[line width=2mm] (0,48) -- (0,42);  
\draw[line width=2mm] (0,30) -- (-30,16);  
\draw[line width=2mm] (0,30) -- (30,24);  
\draw[line width=2mm] (30,9) -- (30,3);   
\draw[line width=2mm] (30,-8) -- (0,-20);  
\draw[line width=2mm] (-30,1) -- (0,-20);  
\draw[line width=2mm] (0,-34) -- (0,-40);  
\draw[line width=2mm] (0,-54) -- (0,-60);  
\end{scope}

\end{tikzpicture}
}
\end{center}
\caption{The smallest non-chain Tutte polynomial poset, the $(6,7)$ Tutte polynomial poset.  (If multiple graphs have the same Tutte polynomial only one graph is shown.)  The graph $\theta(2,2,2) \cdot K_2$ is incomparable with both the graphs $\theta(1,2,4)$ and $C_3 \cdot C_4$.  As one moves upward in the poset, a wide variety of graph parameters associated with the Tutte polynomial increase.  
} 
\label{poset67}
\end{figure}


\begin{thebibliography}{11}

\bibitem{ab84}
A. Agrawal and R. Barlow,
\newblock Survey of network reliability and domination theory.
\newblock \textit{Operations Research}, 32 (1984), 478--492.



\bibitem{b18}
S. Backman,
\newblock Partial graph orientations and the Tutte polynomial.
\newblock \textit{Adv. in Appl. Math.} 94 (2018), 103--119.

\bibitem{bh17}
S. Backman and S. Hopkins,
\newblock Fourientations and the Tutte polynomial.
\newblock \emph{Res. Math. Sci.} 4 (2017), Paper No. 18, 57 pp.

\bibitem{bbst85}
D. Bauer, F. Boesch, C. Suffel, and R. Tindell, 
\newblock Combinatorial optimization problems in the analysis and design of probabilistic networks. 
\newblock \textit{Networks} 15 (1985) 257--271.

\bibitem{bls91}
F.T. Boesch, X. Li, C. Suffel, 
\newblock On the existence of uniformly optimally reliable networks.
\newblock \emph{Networks} 21 (1991), 181--194.


\bibitem{bss09}
F. Boesch, A. Satyanarayana, and C. Suffel,
\newblock A survey of some network reliability analysis and  synthesis results.
\newblock \textit{Networks} 54 (2009), 99--107.





\bibitem{b02}
B. Bollob\'{a}s,
\newblock \textit{Modern Graph Theory},
\newblock Graduate Texts in Mathematics, vol. 184, 
\newblock Springer-Verlag, New York, 2002.


\bibitem{bccgm20}
J. Brown, C. Colbourn, D. Cox, C. Graves, and L. Mol,
\newblock Network reliability: Heading out on the highway.
\newblock \textit{Networks} 77 (2021), 146--160.

\bibitem{bc14}
J. Brown and D. Cox,
\newblock Nonexistence of optimal graphs for all terminal reliability.
\newblock \textit{Networks} 63 (2014), 146--153.




\bibitem{bo92}
T. Brylawski and J. Oxley, 
\newblock The Tutte Polynomial and its Applications.
\newblock in \textit{Matroid Applications}, Encyclopedia of
Mathematics and its Applications, Cambridge University Press, 1992.


\bibitem{c81}
C. Cheng,
\newblock Maximizing the total number of spanning trees in a graph: two related problems in graph theory and optimization design theory.
\newblock \emph{J. Combin. Theory Ser. B} 31 (1981), 240--248.

\bibitem{c87}
C. Colbourn,
\newblock \textit{The Combinatorics of Network Reliability}.
\newblock Oxford University Press (1987).

\bibitem{chm93}
C.~Colbourn, D.~Harms, and W.~Myrvold,
\newblock Reliability polynomials can cross twice. 
\emph{J. Franklin Inst.} 330 (1993), 629--633.








\bibitem{em11}
J. Ellis-Monaghan, C. Merino,
\newblock Graph polynomials and their applications I: The Tutte polynomial.
\newblock \textit{Structural analysis of complex networks}, pgs. 219--255, Birkh\:{a}user/Springer, New York, 2011. 

\bibitem{gs96}
I. Gessel, B. Sagan,
\newblock The Tutte polynomial of a graph, depth-first search, and simplicial complex partitions.
\newblock \textit{Electron. J. Combin.} 3 (1996), no. 2, Research Paper 9, 36pp.




\bibitem{gz83}
C. Green and T. Zaslavsky,
\newblock On the interpretation of Whitney numbers through arrangements of hyperplanes, zonotopes, non-Radon
partitions and orientations of graphs. 
\newblock \textit{Trans. Amer. Math. Soc.} 280 (1983), 97--126.



\bibitem{hpr10}
G. Haggard, D. Pearce, and G. Royle, 
\newblock Computing Tutte polynomials. 
\newblock \textit{ACM Trans. Math. Software} 37 (2010), no. 3, Art. 24, 17 pp. 




\bibitem{k22}
N. Kahl, 
\newblock Extremal graphs for the Tutte polynomial.
\newblock \textit{J. Combin. Theory Ser. B}, 152 (2022), 121--152.



\bibitem{k81}
A.K. Kelmans,
\newblock On graphs with randomly deleted edges.
\newblock \textit{Acta. Math. Acad. Sci. Hung.} 37 (1981), 77--88.

\bibitem{k00}
A.~Kelmans,
\newblock Crossing properties of graph reliability functions. (English summary) 
\newblock \emph{J. Graph Theory} 35 (2000), 206--221. 


\bibitem{l77}
M. Las Vergnas,
\newblock Acyclic and totally cyclic orientations of combinatorial
geometries. 
\newblock \textit{Discrete Math.} 20 (1977), 51--61.

\bibitem{ls24}
L. Landgren and J. Stief,
\newblock A new framework for identifying most reliable graphs and a correction to the $K_{3,3}$-theorem.
\newblock (2024) arxiv preprint arXiv:2407.20217.



\bibitem{m97}
C.~Merino,
\newblock Chip firing and the Tutte polynomial.
\newblock \emph{Ann. Comb.} 1 (1997), 253--259.

\bibitem{mcpp91}
W.~Myrvold, K.~Cheung, L.~Page, and J.A.~Perry,
\newblock Uniformly most reliable networks do not always exist. 
\newblock \emph{Networks} 21 (1991), 417--419.

\bibitem{ow79}
J. Oxley and D.J.A. Welsh,
\newblock The Tutte Polynomial and Percolation.
\newblock  \textit{Graph theory and related topics,} pgs. 329–-339, Academic Press, New York-London, 1979.






\bibitem{r21s}
P. Romero,
\newblock Uniformly optimally reliable graphs: A survey.
\newblock \textit{Networks} 80 (2022), no. 4, 466--481.






\bibitem{s73}
R. Stanley, 
\newblock Acyclic orientations of graphs. 
\newblock \textit{Discrete Math.} 5 (1973), 171--178.




\bibitem{w94}
Z. Wang,
\newblock A proof of Boesch's conjecture.
\newblock \emph{Networks} 25 (1994), 277--284.

\bibitem{w99}
D.J.A. Welsh,
\newblock The Tutte Polynomial.
\newblock Statistical physics
methods in discrete probability, combinatorics, and theoretical
computer science.
\newblock \emph{Random Structures Algorithms} 15 (1999), 210--228.


\end{thebibliography}
\end{document}